\documentclass[12pt]{amsart}
\usepackage{amsfonts, amssymb, latexsym, graphicx, hyperref, amsmath}
\usepackage[vcentermath]{youngtab}
\usepackage{tikz}
\usepackage{tikz-cd}

\usepackage{amssymb}
\usepackage{amsfonts}
\usepackage{amsmath}
\usepackage{bm}
\usepackage{bbm}
\usepackage{enumitem}
\usepackage{esvect}
\usepackage{nicefrac}
\usepackage{tikz}
\usepackage{mathtools}
\usepackage[utf8]{inputenc}
\usepackage{xfrac}
\usepackage{mathtools}
\usepackage{tikz-cd}

\usetikzlibrary{calc, shapes, backgrounds,arrows,positioning,decorations.pathmorphing}

\newtheorem{theorem}{Theorem}[section]
\newtheorem{proposition}[theorem]{Proposition}

\newtheorem{lemma}[theorem]{Lemma}

\newtheorem*{claim*}{Claim}
\newtheorem{corollary}[theorem]{Corollary}
\newtheorem{Main Conjecture}[theorem]{Main Conjecture}

\theoremstyle{definition}
\newtheorem{definition}{Definition}
\theoremstyle{remark}

\theoremstyle{plain}

\newcommand{\cellsize}{11}
\newlength{\cellsz} 
\newsavebox{\cell}
\sbox{\cell}{\begin{picture}(\cellsize,\cellsize)
\put(0,0){\line(1,0){\cellsize}}
\put(0,0){\line(0,1){\cellsize}}
\put(\cellsize,0){\line(0,1){\cellsize}}
\put(0,\cellsize){\line(1,0){\cellsize}}
\end{picture}}
\newcommand\cellify[1]{\def\thearg{#1}\def\nothing{}%
\ifx\thearg\nothing
\vrule width0pt height\cellsz depth0pt\else
\hbox to 0pt{\usebox{\cell} \hss}\fi%
\vbox to \cellsz{
\vss
\hbox to \cellsz{\hss$#1$\hss}
\vss}}
\newcommand\tableau[1]{\vtop{\let\\\cr
\baselineskip -16000pt \lineskiplimit 16000pt \lineskip 0pt
\ialign{&\cellify{##}\cr#1\crcr}}}

\newcommand{\excise}[1]{}

\renewcommand{\Pr}[1]{\mathrm{Pr}\left[#1\right]}
\newcommand{\E}[1]{\mathbb{E}\left[#1\right]}
\newcommand{\V}[1]{\mathrm{Var}\left[#1\right]}
\newcommand{\iv}[1]{#1^{-1}}
\newcommand{\cE}{\mathcal{E}}
\let\gO=\Omega

\title{Proper permutations, Schubert geometry, and randomness}
\author{David Brewster}
\author{Reuven Hodges}
\author{Alexander Yong}
\address{Dept.~of Mathematics, University of Illinois at Urbana-Champaign, Urbana, IL 61801}

\email{davidb2@illinois.edu, rhodges@illinois.edu, ayong@illinois.edu}

\date{December 17, 2020}
\begin{document}
\maketitle

\begin{abstract}
We define and study \emph{proper} permutations. Properness is a geometrically natural necessary criterion for a Schubert variety to be 
Levi-spherical. We prove the probability that a random permutation is proper  goes to zero in the limit.
\end{abstract}

\section{Introduction}

Let $X$ denote the variety of complete flags $\langle 0\rangle \subset F_1 \subset F_2 \subset \cdots \subset
F_{n-1}\subset {\mathbb C}^n$, where $F_i$ is a subspace of dimension $i$. The general linear group $GL_n$
of invertible $n\times n$ complex matrices acts naturally on $X$ by basis change. Let $B\subset GL_n$ be the Borel subgroup
of upper triangular invertible matrices. $B$ acts on $X$ with finitely many orbits; these are the \emph{Schubert cells} $X_w^{\circ}$
indexed by permutations $w$ in the symmetric group $S_n$ on $[n]:=\{1,2,\ldots,n\}$. Their closures $X_w:=\overline{X_{w}^{\circ}}$
are the \emph{Schubert varieties}; these objects are of significant interest in combinatorial algebraic geometry. A standard reference
is \cite{Fulton} and we also point the reader to the expository papers \cite{Gillespie, KL}. 

Now, $\dim{X_w}=\ell(w)$ where 
$\ell(w)=\#\{1\leq i<j\leq n: w(i)>w(j)\}$
counts \emph{inversions} of $w$. Also, let 
\[J(w)=\{1\leq i\leq n-1: \exists \ 1\leq j<i, w(j)=i+1\}\] be the set
of \emph{left descents} of $w$. Assume $I\subseteq J(w)$ and let 
\[D:=[n-1]-I=\{d_1<d_2<\ldots<d_k\};\]
also, $d_0:=0, d_{k+1}:=n$. Let $L_I\subseteq GL_n$ be the Levi subgroup of invertible block diagonal matrices
\[L_I\cong GL_{d_1-d_0}\times GL_{d_2-d_1}\times \cdots \times GL_{d_k-d_{k-1}} \times GL_{d_{k+1}-d_k}.\]
As explained in, \emph{e.g.}, \cite[Section~1.2]{Hodges.Yong}, $L_I$ acts on $X_w$. Moreover, $X_w$ is said to be 
\emph{$L_I$-spherical} if $X_w$ has a dense orbit of a Borel subgroup of $L_I$. If in addition, $I=J(w)$, we say
$X_w$ is \emph{maximally spherical}. We refer the reader to \emph{ibid.}, and the
references therein, for background and motivation about this geometric condition on a Schubert variety.

\begin{definition}
Let $d(w)=\#J(w)$. $w\in S_n$
is  \emph{proper} if $\ell(w)-{d(w)+1\choose 2}\leq n$.
\end{definition}

Actually, for $1\leq n\leq 10$, proper permutations are not rare; the enumeration is:
\[1, 2, 6, 24, 120, 684, 4348, 30549, 236394, 2006492,\ldots\]

Proposition~\ref{prop:Schubsphericalisgood} shows that if $X_w$ is $L_I$-spherical for some $I\subseteq J(w)$, then
$w$ is proper. The proof explains the Lie-theoretic origins of the condition. In this paper, we study proper permutations using
probabilistic considerations.

\begin{theorem}
\label{thm:main2}
If $w\in S_n$ is chosen uniformly at random, ${\rm Pr}[w \text{\ is proper}]\longrightarrow 0$, as $n\to\infty$.
\end{theorem}

Proposition~\ref{prop:Schubsphericalisgood} and Theorem~\ref{thm:main2} combined imply:
\begin{corollary}
\label{thm:main1}
\[\lim_{n\to\infty} {\rm Pr}[w\in S_n, X_w \text{\ is $L_{I}$-spherical for some $I\subseteq J(w)$}]\longrightarrow 0.\] 
In particular,
\[\lim_{n\to\infty} {\rm Pr}[w\in S_n, X_w \text{\ is maximally spherical}]\longrightarrow 0.\] 
\end{corollary}

Theorem~\ref{thm:main2} resolves a conjecture from \cite{Hodges.Yong}. In \emph{ibid.}, the second and
third authors introduced the notion of permutation $w\in S_n$ being \emph{$I$-spherical}; in the case $I=J(w)$ we call $w\in S_n$
\emph{maximally spherical}. This combinatorial definition (recapitulated
in Section~\ref{sec:3}) conjecturally characterizes those $w\in S_n$ such that $X_w$ is $L_I$-spherical. Proposition~\ref{prop:Isphericalisgood} shows that if $w\in S_n$ is $I$-spherical, then $w$ is proper. That proposition, together with
Theorem~\ref{thm:main2}, confirms \cite[Conjecture~3.7]{Hodges.Yong}:

\begin{corollary}
\label{thm:main3}
\[\lim_{n\to\infty} {\rm Pr}[w\in S_n, w \text{\ is $I$-spherical for some $I\subseteq J(w)$}]\longrightarrow 0.\] 
Therefore,
\[\lim_{n\to\infty} {\rm Pr}[w\in S_n, w \text{\ is maximally spherical}]\longrightarrow 0.\] 
\end{corollary} 

Since Corollary~\ref{thm:main3} is consistent with Corollary~\ref{thm:main1}, the former provides additional evidence for the aforementioned
conjectural characterization.

\section{Proof of Theorem~\ref{thm:main2}}\label{sec:2}

For $w\in S_n$, define 
\[\cE_{ij}= \text{\  the event
$\{\iv{w}(i) > \iv{w}(j)\}$.}\]
Let $X_{ij}$ be the indicator for $\cE_{ij}$; that is, $X_{ij}=1$ if event
$\cE_{ij}$ happens and $X_{ij}=0$ otherwise.
Then if $w$ is chosen from $S_n$ uniformly at random, then: \[
  \E{X_{ij}}=\Pr{X_{ij} = 1} = \frac1{2!}(1-\delta_{i,j}) = 1-\Pr{X_{ij}=0}.
\]
Since $\ell(w)=\ell(w^{-1})$ and $\#J(w)=\#\{i:w^{-1}(i+1)<w^{-1}(i)\}$, 
the random variable (r.v.) $\ell(w)-{{d(w)+1} \choose 2}$ can be modeled
as the r.v. 
\[X:=L-{{D+1}\choose 2},\] 
where:
\begin{align*}
  L &= \sum_{i=1}^n\sum_{j=i+1}^n X_{ij},\\
  \text{and } D &= \sum_{i=1}^{n-1}X_{i,i+1}.
\end{align*}
Notice that if $i_1,i_2,i_3,i_4 \in [n]$ are distinct, then $X_{i_1,i_2}$ and
$X_{i_3,i_4}$ are independent.

\begin{lemma}\label{claim:1st-moment}
For $n\geq 2$,
 \[\E{X} = \frac{3n^2-7n+2}{24}.\]
\end{lemma}
\begin{proof}
  It is true that:
  \begin{itemize}
    \item[(a)] $(X_{i,j})_{i<j}$ are identically distributed,
    \item[(b)] $\E{X_{i,i+1}X_{i,i+1}} = \E{X_{i,i+1}^2} = \E{X_{i,i+1}}=1/2$ since
      $X_{i,i+1}$ is an indicator r.v.,
    \item[(c)] $\E{X_{i,i+1}X_{i+1,i+2}}=\Pr{\iv{w}(i) > \iv{w}(i+1) > \iv{w}(i+2)}=\frac{1}{3!}$,
    \item[(d)] $X_{i,i+1}$ and $X_{j,j+1}$ are independent if $i+1 < j$.
  \end{itemize}
  With this, the expression $\E{L}$ can be expanded as:
     \begin{align*}
    \E{L}
    &=\E{\sum_{i=1}^n\sum_{j=i+1}^n X_{ij}} \\
    &=\sum_{i=1}^n\sum_{j=i+1}^n\E{X_{ij}}&\text{lin. of expectation} \\
    &=\frac{1}{2}{n\choose 2}&\text{identically distributed}. \\
  \end{align*}
  Similarly, 
  \[\E{D}=\frac{n-1}2.\]
  Next, the expression $\E{D^2}$ can be expanded as:
  \begin{align*}
    \E{D^2}
    &=\E{\left(\sum_{i=1}^{n-1}X_{i,i+1}\right)^2}\\
    &=\E{\sum_{i=1}^{n-1}X_{i,i+1}^2 + \sum_{i=1}^{n-1}\sum_{j\neq i}X_{i,i+1}X_{j,j+1}}\\
        &=\sum_{i=1}^{n-1}\E{X_{i,i+1}^2} + \sum_{i=1}^{n-1}\sum_{j\neq i}\E{X_{i,i+1}X_{j,j+1}}&\text{lin. of expectation}\\
     &=\frac{n-1}2 + \sum_{i=1}^{n-1}\sum_{j\neq i}\E{X_{i,i+1}X_{j,j+1}}&\text{by (b)}\\
      &=\frac{n-1}2 + 2\sum_{i=1}^{n-1}\sum_{j=i+1}^{n-1}\E{X_{i,i+1}X_{j,j+1}}\\
      &=\frac{n-1}2 + 2\left(\sum_{i=1}^{n-2}\E{X_{i,i+1}X_{i+1,i+2}} + \sum_{i=1}^{n-1}\sum_{j=i+2}^{n-1}\E{X_{i,i+1}X_{j,j+1}}\right)\\
    &=\frac{n-1}2 + 2\left(\frac{n-2}{3!} + \frac{1}{2^2}\left({{n-1}\choose 2}-(n-2)\right)\right)&\text{by (c) and (d))}\\
    &=\frac{n-1}2 + \frac{n-2}{3} + \frac{1}{2}\left({{n-1}\choose 2}-(n-2)\right).
  \end{align*}
  Thus by linearity of expectation, 
  \[\E{{{D+1}\choose 2}}=\frac12\E{D^2+D} = \frac{n-1}{2}+\frac{n-2}{6} -\frac{n-2}{4}+\frac14{{n-1}\choose 2}\] and 
  \begin{align*}
    \E{X}
    &=\E{L-{{D+1}\choose 2}}\\
    &=\E{L} - \E{{{D+1}\choose 2}}\\
    &= \frac{3n^2-7n+2}{24}.\qedhere
  \end{align*}
\end{proof}

\begin{lemma}\label{claim:2nd-moment}
  \[\E{X^2} = \frac{n^4}{64} + o(n^4).\]
\end{lemma}

\begin{proof}
\noindent
Notice that: \begin{align*}
  \E{X^2}
  &=\E{L^2}+\E{{{D+1} \choose 2}^2}-2\E{L{{D+1}\choose 2}} \\
  &=\E{L^2}+\frac14\left(\E{D^4}+2\E{D^3}+\E{D^2}\right)-\E{LD^2}-\E{LD}.
\end{align*}
Now, $0\leq D^3, D^2, LD\leq n^3$, so $\E{D^3},\E{D^2},\E{LD}= o(n^4)$.
Thus it suffices to study the asymptotics of $\E{L^2}, \E{D^4/4}, \E{LD^2}$.

We will repeatedly use the following observation. For a set $S$ with $|S| = o(f(n))$: 
\begin{equation}
\label{eqn:settrick}
  \sum_{(i_1,j_1,\ldots,i_c,j_c)\in S}\E{\prod_{k=1}^cX_{i_k,j_k}}\leq |S| = o(f(n)).
\end{equation}

Expanding $\E{L^2}$ gives: \[
      \E{L^2} = \sum_{i=1}^n\sum_{j=i+1}^n\sum_{i'=1}^n\sum_{j'=i'+1}^n\E{X_{i,j}X_{i',j'}}.
    \]
    There are ${n\choose 2}^2 = n^4/4 + o(n^4)$ many terms
    in this summation.
    Further, there are ${n\choose 2}{{n-2}\choose 2}=n^4/4 + o(n^4)$ many terms
    in this summation such that $i,j,i',j'$ are distinct.
    Therefore, there must be $o(n^4)$ terms where $i,j,i',j'$ are \emph{not}
    distinct.    Now, 
    \begin{align*}
    &  \ \ \ \ \ \sum_{\text{distinct }i<j,i'<j' \in [n]}\E{X_{i,j}X_{i',j'}}\\
     &= \sum_{\text{distinct }i<j,i'<j' \in [n]}\E{X_{i,j}}\E{X_{i',j'}}&\text{(independence when indices are distinct)} \\
      &=\left(\frac12\right)^2 {n\choose 2}{{n-2}\choose 2} \\
      &=\left(\frac12\right)^2(n^4/4 + o(n^4)).
    \end{align*}
Combining this with (\ref{eqn:settrick}) gives
\begin{equation}
\label{eqn:first}
\E{L^2}=\frac{1}{16}n^4+o(n^4).
\end{equation}

  To expand $\E{D^4/4}$, first we have \[
      \E{D^4} = \sum_{i=1}^{n-1}\sum_{j=1}^{n-1}\sum_{i'=1}^{n-1}\sum_{j'=1}^{n-1}\E{X_{i,i+1}X_{j,j+1}X_{i',i'+1}X_{j',j'+1}}.
    \]
    There are $(n-1)^4 = n^4 + o(n^4)$ many terms in this
    summation.
    Further, there are $4!{{n-4}\choose 4}=n^4 + o(n^4)$ many terms
    in this summation such that $i,i+1,j,j+1,i',i'+1,j',j'+1$ are distinct.
    Here we have used the fact that there are ${{n-k}\choose k}$ ways to choose
    $k$ non-consecutive numbers from $[n-1]$.
    Therefore, there must be $o(n^4)$ terms where $i,i+1,j,j+1,i',i'+1,j',j'+1$
    are \emph{not} distinct.    We compute
    \begin{align*}
      & \ \ \ \ \ \frac14\cdot\sum_{\substack{i,j,i',j' \in [n]\\i,i+1,j,j+1,i',i'+1,j',j'+1\text{ are distinct}}}\E{X_{i,i+1}X_{j,j+1}X_{i',i'+1}X_{j',j'+1}}\\
      &= \frac14\cdot\sum_{\substack{i,j,i',j' \in [n]\\i,i+1,j,j+1,i',i'+1,j',j'+1\text{ are distinct}}}\E{X_{i,i+1}}\E{X_{j,j+1}}\E{X_{i',i'+1}}\E{X_{j',j'+1}}\\
      &= \frac14\cdot\left(\frac12\right)^4\cdot4!{{n-4}\choose 4} \\
      &= \frac14\cdot\left(\frac12\right)^4(n^4 + o(n^4)). 
          \end{align*}
    Hence by (\ref{eqn:settrick}),
    \begin{equation}
    \label{eqn:second}
    \E{D^4/4}=\frac{1}{64}n^4+o(n^4).
    \end{equation}

   Expanding $\E{LD^2}$ gives: \[
      \E{LD^2} = \sum_{i=1}^n\sum_{j=i+1}^n\sum_{i'=1}^{n-1}\sum_{j'=1}^{n-1}\E{X_{i,j}X_{i',i'+1}X_{j',j'+1}}.
    \]
    There are ${n\choose 2}(n-1)^2 = n^4/2 + o(n^4)$ many
    terms in this summation.
    Further, there are $2!{{n-2}\choose 2}{{n-4}\choose 2}=n^4/2 + o(n^4)$ many
    terms such that $i,j,i',i'+1,j',j'+1$ are distinct.
    This can be seen by first choosing $i'$ and $j'$, and then choosing the pair
    $(i,j)$ such that $i < j$.
    Therefore, there must be $o(n^4)$ terms where $i,j,i',i'+1,j',j'+1$ are
    \emph{not} distinct.    We have:
    \begin{align*}
      & \ \ \ \ \  \sum_{\substack{i<j,i',j' \in [n]\\i,j,i',i'+1,j',j'+1\text{ are distinct}}}\E{X_{i,j}X_{i',i'+1}X_{j',j'+1}}\\
            &= \sum_{\substack{i<j,i',j' \in [n]\\i,j,i',i'+1,j',j'+1\text{ are distinct}}}\E{X_{i,j}}\E{X_{i',i'+1}}\E{X_{j',j'+1}}
            \end{align*}
            \begin{align*}
                       &= \left(\frac12\right)^3\cdot 2!{{n-2}\choose 2}{{n-4}\choose 2} \\
      &= \left(\frac12\right)^3\cdot(n^4/2 + o(n^4)).
    \end{align*}
    Therefore by (\ref{eqn:settrick}), 
    \begin{equation}
    \label{eqn:third}
    \E{LD^2}=\frac{1}{16}n^4+o(n^4).
    \end{equation}

Summarizing, we have shown that
\[
  \E{X^2}=
  \E{L^2} + \E{D^4/4} - \E{LD^2} +o(n^4).\]
  Now the result follows from (\ref{eqn:first}), (\ref{eqn:second}), (\ref{eqn:third}).
\end{proof}

\begin{lemma}
 $\lim_{n\to \infty}\Pr{X \leq n}\longrightarrow 0$.
\end{lemma}
\begin{proof}
  The event $\{X\leq n\}$ is contained in the event $\{|X-\E X|\geq t\}$
  when $t=\E{X}-n$ because $|X-\E X|\geq t$ implies that either
  \begin{itemize}
    \item[(A)] $X-\E X \geq t$, or
    \item[(B)] $\E X - X \geq t$,
  \end{itemize}
  and the above choice of $t$ causes inequality (B) to be $X\leq n$.
  Now, we can apply Chebyshev's Inequality to $X$ and $t=\E X -n$ to
  get:
  \begin{align*}
    \Pr{X \leq n}
    &\leq \Pr{|X-\E{X}|\geq \E X-n} \\
    &\leq \frac{\V X}{(\E X-n)^2} \\
    &=\frac{\E{X^2}-(\E{X})^2}{(\E X-n)^2}.
  \end{align*}
  The result follows from the fact that, by Lemma~\ref{claim:2nd-moment},
  \[\E{X^2} = \frac{n^4}{64} + o(n^4)\]
  and by  Lemma~\ref{claim:1st-moment}, both
  \[(\E{X})^2 = \frac{n^4}{64} + o(n^4) \text{ \ and $(\E{X}-n)^2 = \gO(n^4)$}.\qedhere\]
\end{proof}
This completes the proof of Theorem~\ref{thm:main2}. \qed

\section{Properness is necessary for sphericality; proof of Corollaries~\ref{thm:main2} and~\ref{thm:main3}}\label{sec:3}

 Let $T$ be the maximal torus of diagonal matrices in $GL_n$. For $I \subseteq J(w)$, define 
\[B_I = L_I \cap B.\] Hence $B_I$ is the Borel subgroup of upper triangular matrices in $L_I$. For a positive integer $j$, let $U_j$ be the maximal unipotent subgroup of $GL_j$ consisting of upper triangular matrices with 1's on the diagonal. Then
\begin{equation}
\label{eq:dimUnipotent}
\dim{U_j} = {j \choose 2}.
\end{equation}
Let $U_I$ be the maximal unipotent subgroup of $B_I$. It is basic (see, \emph{e.g.}, \cite[Chapter~IV]{Borel}) that
\begin{equation}
\label{eq:formUnipotent}
U_{I} \cong U_{d_1-d_0}\times U_{d_2-d_1}\times \cdots \times U_{d_k-d_{k-1}} \times U_{d_{k+1}-d_k}.
\end{equation}

\begin{proposition}
\label{prop:Schubsphericalisgood}
If $X_w$ is $L_I$-spherical then $w$ is proper.
\end{proposition}
\begin{proof}
Since $L_I$ acts spherically on $X_w$, by definition, there is a Borel subgroup $K \subset L_I$ such that $K$ has a dense orbit ${\mathcal O}$ in $X_w$. Thus \[\dim{X_w} = \dim{{\mathcal O}}.\] Let $x\in {\mathcal O}$. By \cite[Proposition 1.11]{Brion}, ${\mathcal O} = K \cdot x$ is a smooth, closed subvariety of $X_w$ of dimension $\dim{K} - \dim{K_x}$, where $K_x$ is the isotropy group of $x$. Hence
\begin{equation}
\label{eq:dimUpperBoundSPherical}
\dim{X_w} = \dim{{\mathcal O}} = \dim{K}-\dim{K_x} \leq \dim{K}.
\end{equation}

All Borel subgroups of a connected algebraic group are conjugate~\cite[\S 11.1]{Borel}, and so $\dim{K}=\dim{B_I}$. The fact that $L_I$ acts on $X_w$ implies $I \subseteq J(w)$, and hence $L_I \subseteq L_{J(w)}$~\cite[Section~1.2]{Hodges.Yong}. This implies $B_I \subseteq B_{J(w)}$. By \cite[Theorem 10.6.(4)]{Borel}, $B_I=T\ltimes U_I$. Combining all this we have
\begin{equation}
\label{eq:dimUpperBoundSPherical2}
\dim{K} = \dim{B_I} \leq \dim{B_{J(w)}} = \dim{T} + \dim{U_{J(w)}}.
\end{equation}

Let $D = [n-1]-J(w)=\{d_1<d_2<\ldots<d_k\}$. It follows from \eqref{eq:dimUnipotent} and \eqref{eq:formUnipotent} that
\begin{equation*}
\dim{U_{J(w)}} = {d_1-d_0 \choose 2} + {d_2-d_1 \choose 2} + \cdots + {d_{k+1}-d_k \choose 2}.
\end{equation*}
The right hand side is maximized when there exists a $t$ such that $d_t - d_{t-1} = n - k$ and $d_j - d_{j-1} = 1$ for all $j \neq t$. Thus
\begin{equation*}
\dim{U_{J(w)}} \leq {n-k \choose 2} = {n-((n-1)-d(w))  \choose 2} = {d(w) + 1 \choose 2}.
\end{equation*}
Combining this with \eqref{eq:dimUpperBoundSPherical}, \eqref{eq:dimUpperBoundSPherical2}, and the fact that $\ell(w) = \dim{X_w}$, we see $\ell(w) \leq n + {d(w) + 1 \choose 2}$, that is, $w$ is proper.
\end{proof}

Next, we recall the definition of $I$-spherical permutations in $S_n$ \cite{Hodges.Yong}. Let $s_i=(i \ i+1)$ denote the
simple transposition interchanging $i$ and $i+1$. An expression $w=s_{i_1}s_{i_2}\cdots s_{i_{\ell}}$ for $w\in S_n$
is \emph{reduced} if $\ell=\ell(w)$. Let ${\sf Red}(w)$ be the set of all reduced expressions for $w$.

\begin{definition}[Definition~3.1 of \cite{Hodges.Yong}]\label{def:Ispherical}
$w\in S_n$ is \emph{$I$-spherical} 
if $R=s_{i_1}s_{i_2}\cdots s_{i_{\ell(w)}}\in {\sf Red}(w)$ exists such that
\begin{itemize}
\item[(I)] $s_{d_i}$ appears at most once in $R$
\item[(II)] $\#\{m:d_{t-1}<i_m< d_t\} \leq {d_{t}-d_{t-1}+1\choose 2}-1$   for $1\leq t\leq k+1$.
\end{itemize}
\end{definition}

This is a combinatorial analogue of Proposition~\ref{prop:Schubsphericalisgood}:

\begin{proposition}
\label{prop:Isphericalisgood}
Let $w\in S_n$ and $I\subseteq J(w)$.
If $w$ is $I$-spherical then $w$ is proper.
\end{proposition}
\begin{proof}
First suppose $I=J(w)$. Consider a reduced word $R\in {\sf Red}(w)$. By Definition~\ref{def:Ispherical}(I), at most $n-1-d(w)$
of the factors of $R$ are of the form $s_x$ where $x\not\in J(w)$.
Thus, at least $\ell(w)-(n-1-d(w))$ factors are of the form $s_x$ where $x\in J(w)$.
Clearly, if $j_1,...j_k$ are positive integers then
$\sum_{i=1}^{k+1} {j_i+1\choose 2} \leq {j_1+...+j_{k+1}+1 \choose 2}$. 
Equivalently,
\[\sum_{i=1}^{k+1} {j_i+2\choose 2} - 1 = \sum_{i=1}^{k+1} {j_i+1\choose 2} + j_i \leq {j_1+...+j_{k+1}+1 \choose 2}+(j_1+...+j_{k+1}).\]

Set $j_i=d_{i}-d_{i-1}-1$ (for $1\leq i\leq k+1$). Then $j_1+\ldots +j_{k+1}=d_{k+1}-d_0-(k+1)=n-1-k=d(w)$. Thus, by 
Definition~\ref{def:Ispherical}(II), at most 
${d(w)+1 \choose 2} +d(w)$ factors are of the from $s_x$ where $x\in J(w)$.
Therefore, 
\[{d(w)+1\choose 2}+d(w)\geq \ell(w)-(n-1-d(w)).\] 
Rearranging,
$\ell(w)\leq n-1-d(w)+{d(w)+1\choose 2}+d(w)\!\!\iff\!\! \ell(w)<n+{d(w)+1\choose 2}$. So,
$w$ is proper.

For $I\neq J(w)$, we use that if $w$ is $I$-spherical then $w$ is $J(w)$-spherical
\cite[Proposition~2.12]{Hodges.Yong}.
\end{proof}

\noindent
\emph{Conclusion of proof of Corollaries~\ref{thm:main1} and~\ref{thm:main3}:} These claims follow immediately
 from Theorem~\ref{thm:main2} combined with Proposition~\ref{prop:Schubsphericalisgood} and Proposition~\ref{prop:Isphericalisgood}, respectively.\qed

Although we chose not to pursue it, using similar techniques, it should be possible to prove analogues of our results for the other classical Lie types.
 
\section*{Acknowledgements}
We thank Mahir Can and Yibo Gao for helpful discussions.
The project was completed
as part of the ICLUE (Illinois Combinatorics Lab for Undergraduate Experience) program, which was funded by the NSF RTG grant DMS 1937241. 
AY was partially supported by a Simons Collaboration grant, and
UIUC's Center for Advanced Study. RH was partially supported by an AMS Simons Travel grant.

\end{document}